\documentclass[12pt]{amsart}
\usepackage{amssymb}
\usepackage{amsmath}
\usepackage{version}
\usepackage{color}
\usepackage{amsfonts}
\numberwithin{equation}{section}
\newtheorem{theorem}{Theorem}[section]
\newtheorem{lemma}[theorem]{Lemma}
\newtheorem{corollary}[theorem]{Corollary}
\newtheorem{example}[theorem]{Example}

\newtheorem{definition}[theorem]{Definition}
\newtheorem{proposition}[theorem]{Proposition}

\newtheorem{remark}[theorem]{Remark}

\newcommand\beq{\begin{equation}}

\newcommand\de{\delta}
\newcommand\eeq{\end{equation}}
\newcommand\re{\mathrm {Re~}}
\newcommand\im{\mathrm {Im~}}
\newcommand\ii{\mathrm i}

\newcommand\D{\mathbb D}

\newcommand\C{\mathbb C}

\newcommand\R{\mathbb R}

\newcommand\e{\mathrm e}
\newcommand\Pick{\mathcal P}

\newcommand\df{\stackrel{\rm def}{=}}

\newcommand\ont{o_{\rm nt}}
\newcommand\nt{\stackrel{\mathrm {nt}}{\to}}
\newcommand\nn{\nonumber}

\DeclareMathOperator{\schur}{schur}
\DeclareMathOperator\diag{diag}
\DeclareMathOperator\rank{rank}
\DeclareMathOperator\dist{dist}
\let\phi\varphi
\begin{document}
\title{Pseudo-Taylor expansions and the Carath\'{e}odory-Fej\'{e}r  problem}

\author{Jim Agler, Zinaida A. Lykova and N. J. Young}

\date{6 January 2011}

\begin{abstract} We give a new solvability criterion for the {\em  boundary Carath\'{e}odory-Fej\'{e}r problem}: given a  point $x \in \R$ and,  a finite set of target values $a^0,a^1,\dots,a^n \in \R$, to construct a function $f$ in the Pick class such that
the limit of $f^{(k)}(z)/k!$ 
as $z \to x$ nontangentially in the upper half plane is $a^k$.  The criterion is in terms of positivity of an associated Hankel matrix.  The proof is based on a reduction method due to Julia and Nevanlinna. 
\end{abstract}

\thanks {Jim Agler was partially supported by National Science Foundation Grant DMS 0801259. N. J. Young was partially supported by London Mathematical Society Grant 4918 and by EPSRC Grant EP/G000018/1.}


\maketitle
\markboth{Jim Agler, Zinaida A. Lykova and N. J. Young}
{Pseudo-Taylor expansions and the  Carath\'{e}odory-Fej\'{e}r  problem}

\section{Introduction} \label {intro}

A theme of classical analysis is to ascertain whether a given finite sequence of complex numbers comprises the initial Taylor coefficients of an analytic function of a specified class on a domain $U$ about some point $x$ of $U$.  In the case that $U$ is the upper halfplane
\[
\Pi \df \{ z\in\C: \im z > 0 \}
\]
and the specified class is the  Pick class $\Pick$ we obtain the much-studied {\em Carath\'{e}odory-Fej\'{e}r problem} \cite{CF, bgr}.  Here
$\Pick$ is defined to be the set of analytic functions $f$ on $\Pi$ such that $\im f \geq 0$ on $\Pi$.  We can equally well ask the same question for a point $x \in \partial U$, the boundary of $U$, but then, since an analytic function on $U$ will not 
in general have a Taylor expansion about every point in $\partial U$, there is a question as to how we should interpret ``Taylor coefficients".  The simplest answer is just to restrict attention to functions which are analytic at the interpolation point $x$.  We then arrive at the {\em boundary Carath\'{e}odory-Fej\'{e}r problem}:\\

\noindent {\bf Problem $\partial CF\Pick(\R)$} \quad {\em
Given a point $x\in \R$ and $a^0,a^1,\dots,a^n \in \R$, find a function $f$ in the Pick class such that $f$ is analytic at  $x$ and
\beq \label{interpcond_anal}
\frac{f^{(k)}(x)}{k!} = a^k, \qquad k=0,1,\dots,n.
\eeq
}

In \cite{ALY10} we gave a new criterion for Problem $\partial CF\Pick(\R)$ to have a solution $f$.  Roughly speaking, such an $f$ exists if and only if a certain Hankel matrix constructed from the $a^k$ is either positive definite or {\em southeast-minimally positive} (to be defined in Section \ref{weak} below).    However,  the requirement that the solution $f$ be analytic at $x$ is unnecessarily strong.   There is a natural weaker notion of solution,
in which the role of Taylor coefficients is played by a generalization appropriate to points on the boundary of $\Pi$.  We call these {\em pseudo-Taylor coefficients}\footnote{Although this notion has been in use for 90 years or more, we cannot find an agreed name for it, and so are herewith introducing one.}.
In this paper we show that the same existence criterion applies to solvability in this weaker sense.  It follows that the problem has a solution in the original (analytic at $x$) sense if and only if it has a solution in the weaker sense.

As in \cite{ALY10}, our main tool is a technique of reduction of functions in the Pick class that is originally due to G. Julia \cite{Ju20} and was greatly strengthened by Nevanlinna \cite{Nev1922}.  One virtue of this approach is that it is elementary; it does not depend on the theories of operators or Hilbert function spaces.  The proof is by induction on $n$ in combination with Julia reduction and an identity for Hankel matrices.  A point of the paper is that the methods of \cite{ALY10} remain valid for the more delicate problems associated with weak solutions.  We shall use some results from that paper, and for convenience we often refer to \cite{ALY10} for proofs even of statements that are well established.

The present paper could be regarded as a correction of Nevanlinna's own treatment: we rectify an oversight that led him to an incorrect statement about solvability.   A discussion of Nevanlinna's assertion and a counterexample are given in \cite[Section 10]{ALY10}.

The problem we study has an extensive history, which we discuss in \cite[Sections 1 and 10]{ALY10}. We mention in particular a very recent paper \cite{Bol10}.

The paper is organised as follows. In Section \ref{taylor}
we define weak solutions and the notion of a pseudo-Taylor expansion of $f \in \Pick$ about $x \in \R$.
In Section \ref{reduction} we describe Julia's reduction procedure and its inverse, and give important properties of  these procedures. 
In Section \ref{relax} we show that positivity of a Hankel matrix is necessary and sufficient for weak solvability of a relaxation of Problem $\partial CF\Pick(\R)$, in which the last of the interpolation conditions (\ref{interpcond_anal}) is relaxed (equality is replaced by an inequality).  
In Section \ref{weak}  we prove our main theorem:  Problem $\partial CF\Pick$ has a solution in the weak sense if and only if its associated Hankel matrix is either positive definite or southeast-minimally positive.  

We shall write the imaginary unit as $\ii$, in Roman font, to have $i$ available for use as an index.   We denote the open unit disc by $\D$.

 
\section{Pseudo-Taylor expansions} \label{taylor}
  Recall that for any domain $U$ and any $x\in\partial U$, a subset $S$ of $U$ {\em approaches $x$ nontangentially} if $x$ is in the closure of $S$ and the quotient $|z-x|/\dist(z, \partial U)$ is bounded for $z\in S$, and that $z\to x$ {\em nontangentially in $U$} if $z\to x$ and $z$ lies in some set $S$ that approaches $x$ nontangentially.
We shall use the notation $z\nt x$ to mean that $z \to x$ nontangentially in a given domain $U$.
 For a function 
$h$ analytic on $U$ and a positive integer $n$, we write 
\[ 
h(z) = \ont ((z-x)^{n}) 
\]
to mean that $\frac{h(z)}{(z-x)^{n}} \to 0$ as $z \nt x$.   A function $f$ analytic on $U$ will be said to have a {\em pseudo-Taylor expansion of order $n$ about $x\in\partial U$} if there exist $c^0, c^1, \dots,c^n \in \C$ such that 
\beq \label{psn}
 f(z) = c^0 + c^1 (z-x) + \dots + c^n (z-x)^{n} + \ont ((z-x)^{n}).
\eeq
We call $c^j$ the {\em  $j$th pseudo-Taylor coefficient} of $f$ at $x$.  In fact pseudo-Taylor coefficients are not well defined in complete generality (for example, if $U$ has a cusp at $x$), but it is easy to see that if there is a line segment in $U$ that approaches $x$ nontangentially then $c^j$ is uniquely determined.  In particular, pseudo-Taylor expansions of a function in $\Pick$ about a point $x\in\R$ are unique when they exist.

Pseudo-Taylor expansions are of course a form of asymptotic expansion, but are sufficiently special to deserve a separate name.

Note the special case $n=1$: $f$ has a pseudo-Taylor expansion of order $1$ if and only if $f$ has a nontangential limit $a^0$ and an angular derivative $a^1$ at $x\in\partial U$, and then the expansion of $f$ is $a^0 + a^1(z-x) + \ont(z-x)$.  See \cite{car54, S2} for the notion of angular derivative.  Pseudo-Taylor coefficients can thus be regarded as generalizations of angular derivatives.

We define a function $f \in \Pick$ to be a {\em weak solution} of Problem $\partial CF\Pick(\R)$ if $f$ has a pseudo-Taylor expansion of order $n$ at $x$ and the $j$th pseudo-Taylor coefficient of $f$ at $x$ is $a^j$ for $j=0,1,\dots,n$.  Thus $f$ is a weak solution if and only if 
\beq \label{psTay_int}
 f(z) = a^0 + a^1 (z-x) + \dots  + a^n (z-x)^{n} + R_n(z)
\eeq
where
\beq \label{psTay2_int}
 \frac{R_n(z)}{(z-x)^{n}} \to 0
\eeq
as $z \to x$ nontangentially in $\Pi$.

This is essentially the notion of solution used by R. Nevanlinna \cite{Nev1922}; we believe he was the first mathematician to study such problems\footnote{Actually Nevanlinna took the interpolation node $x$ to be $\infty$.}, and subsequent authors (e.g. \cite{BD,BolKh08,Geo05}) have used equivalent notions.

Pseudo-Taylor expansions behave very differently from Taylor expansions, as the following examples show.
\begin{example}\label{ex1}
\rm The function $f(z)= z/(1-z\log z)$ is in $\Pick$ and has the pseudo-Taylor expansion
\[
f(z) = z + \ont(z)
\]
but no pseudo-Taylor expansion of order $2$ about $0$.
\end{example}
\begin{example} \label{ex2} \rm
Let $\nu$ be a positive integer, $\nu \ge 4$, and let
$$
f_{\nu}(z)= - \sum_{k=1}^{\infty} \frac{1}{k^{\nu} z + k^{\nu-1}}, \;\; z \in \Pi.
$$
Then $f_{\nu}\in\Pick$ has the pseudo-Taylor expansion 
\beq\label{expandfnu}
f_{\nu}(z)= -\zeta(\nu-1) +\zeta(\nu -2)z - \dots + (-1)^{\nu} \zeta(2)z^{\nu -3} +  \ont(z^{\nu-3})
\eeq
of order $\nu-3$ about $0$, but has no pseudo-Taylor expansion of order $\nu-2$.
We justify this assertion in the Appendix.
\end{example}
\begin{example}\label{ex3} \rm
The function
\[
f(z) = -\frac{1}{\mathrm e}\sum_{k=1}^\infty \frac{1}{k!(z+\frac{1}{k})}
\]
is in $\Pick$ and has a pseudo-Taylor expansion of infinite order about $0$, to wit
\[
f(z) = -1+ 2z -5z^2 + 15z^3-52z^4+ \dots =\sum_{n=0}^\infty (-1)^{n+1} A_n z^n
\]
where $A_0=1$ and, for $n\geq 0$,
\[
A_{n+1} = 2A_n + \sum_{r=1}^n {n \choose r} A_{n-r}.
\]
That is, $f$ has
pseudo-Taylor expansions of all orders about $0$, but $f$ is clearly not analytic at $0$.
\end{example}

There are two other natural ways that a function $f$ could be regarded as a solution of a boundary interpolation problem without necessarily being analytic at the interpolation node.  Let us say that $f\in\Pick$ is a {\em nontangential solution} of Problem $\partial CF\Pick$
if
\[
\lim_{z\nt x} \frac{f^{(k)}(z)}{k!} = a^k \qquad \mbox { for } k=0,1,\dots,n,
\]
and is a {\em radial solution} of Problem $\partial CF\Pick$ if
\[
\lim_{y\to 0+} \frac{f^{(k)}(x+\ii y)}{k!} = a^k \qquad \mbox { for } k=0,1,\dots,n.
\]
Fortunately it transpires that the notions of weak, nontangential and radial solution all coincide.
The following theorem is widely known; see for example \cite[VI-1]{S2} and \cite[Corollary 7.9]{BD} for the corresponding statement for functions analytic on the open disc $\D$.

\begin{theorem}\label{Taylor} Let  $f$ be a function 
 analytic on $\Pi$, let $x \in \R$, let $n$ be a non-negative integer and let $a^j \in \C$ for $j=1,2, \dots, n$. Then the following statements are equivalent:
\begin{enumerate}
\item[\rm (i)] $f$ has the pseudo-Taylor expansion 
\beq \label{pseudo-Taylor}
 f(z) = a^0 + a^1 (z-x) + \dots  + a^n (z-x)^{n} + \ont ((z-x)^n), 
\eeq
of order $n$ about $x$;

\item[\rm (ii)]  the derivatives $f^{(k)}$, $k=0,1,\dots,n$, have nontangential limits at $x$ and
\beq \label{ang_derv}
\lim_{z \nt x}\frac{f^{(k)}(z)}{k!} = a^k, \qquad k=0,1,\dots,n; \nn
\eeq

\item[\rm (iii)]  the derivatives $f^{(k)}$, $k=0,1,\dots,n$, have radial limits at $x$ and
\beq \label{rad_derv}
\lim_{y\to 0+}\frac{f^{(k)}(x+\ii y)}{k!} = a^k, \qquad k=0,1,\dots,n. \nn
\eeq
\end{enumerate}
\end{theorem}

\begin{proof} 
(ii) $\Rightarrow $(iii) is trivial.  We prove
(iii) $\Rightarrow $(i). 
The statement is true when $n=0$: this is precisely Lindel\"of's Principle \cite[Theorem 8.7.1]{Kr}. Let us assume it is true for $n-1$, where $n \ge 1$, and deduce that it holds for $n$.

Suppose that (iii) holds, and let $g = f'$. Then 
\[
\lim_{y\to 0+}\frac{g^{(k)}(x+\ii y)}{k!}=\lim_{y\to 0+}\frac{f^{(k+1)}(x+\ii y)}{k!} = (k+1) a^{k+1}
\]
for $k=0,1,\dots,n-1$. By the inductive hypothesis applied to $g$,
for  $z \in \Pi$,
\beq \label{n-1-pseudo-Taylor}
f'(z) =g(z)=
a^1  + 2 a^{2} (z-x) + \dots  + n a^n  (z-x)^{n-1} + \ont ((z-x)^{n-1}) \nn
\eeq
and hence
\beq \label{n-pseudo-Taylor} \nn
f'(z) - \left(
a^1  + 2 a^{2} (z-x) + \dots  + n a^n  (z-x)^{n-1} \right) = \beta(z)(z-x)^{n-1},
\eeq
for some function $\beta$ such that  $\beta(z)  \to 0$ as $z \nt x$.
We denote by $[x,z]$ the straight line segment joining $x$ and  $z$ in $\Pi$.
Now 
\begin{eqnarray}\label{0_derv}
&~&\frac{f(z) - a^0 - a^1 (z-x) - \dots - a^n (z-x)^{n} }{(z-x)^{n}} \nn \\
&=& \frac{1}{(z-x)^n} \; \int_{[x,z]} f'(\zeta) -  a^1 - 2 a^{2} (\zeta-x)- \dots  - a^n n (\zeta-x)^{n-1}   ~d \zeta  \nn \\
&=& \frac{1}{(z-x)} \; \int_{[x,z]}  \beta(\zeta) \left(\frac{\zeta-x}{z-x}\right)^{n-1}~d \zeta. 
\end{eqnarray}
The right hand side of (\ref{0_derv})  tends to $0$ as $z$ tends nontangentially to $x$.
Hence
\[
\lim_{z \nt x} \frac{f(z) -( a^0 + a^1 (z-x) + \dots  + a^n (z-x)^{n} )}{(z-x)^{n}}=0.
\]
The statement follows by induction.

(i) $\Rightarrow$ (ii)~ Let $K >0$ and consider the nontangential approach region at $x$
\[
S_K = \{ z \in \Pi: |z -x| \le K \;\dist (z, \R) \}.
\]
For $z$ in $S_K$ let $\gamma_z$ denote the circle with center $z$ and radius $\tfrac{1}{2}\dist (z, \R)$. It is clear that 
$\gamma_z$ lies in $\Pi$. Note that, for each $\zeta \in \gamma_z$,
we have $|\zeta -z|=\tfrac{1}{2} \dist (z, \R)$, and so
$$ \tfrac{1}{2}\;\dist (z, \R) \le \dist (\zeta, \R) \le \tfrac{3}{2}\;\dist (z, \R).$$
Thus, for each $\zeta \in \gamma_z$,
\begin{eqnarray} \label{gamma_z_to _x} \nn
|\zeta -x| & \le & |\zeta -z|+|z -x| = \tfrac{1}{2} \;\dist (z, \R) + |z -x| \le \tfrac{3}{2}|z -x|
\end{eqnarray}
and so
\begin{eqnarray} \label{on_gamma_z}
|\zeta -x| & \le & |\zeta -z|+|z -x| \le \tfrac{1}{2} \;\dist (z, \R) +  K \;\dist (z, \R) \\
 & \le &(2K +1) \; \dist (\zeta, \R) \nn.
\end{eqnarray}
Thus $\gamma_z$ lies in $S_{2K+1}$.

By equation (\ref{pseudo-Taylor}), for any $\zeta \in U$,
\beq \label{pseudo-Taylor-gamma} \nn
 f(\zeta) = a^0 + a^1 (\zeta-x) + \dots + a^{n-1} (\zeta-x)^{n-1} + a^n (\zeta-x)^{n} + \beta(\zeta)(\zeta-x)^{n},
\eeq
where $\beta(\zeta) \to 0$ as $\zeta \nt x$.
For $0 < \varepsilon <1$, define the number $\alpha (\varepsilon)$ by
$$
\alpha (\varepsilon) = \sup \left\{ \left| \beta(\zeta)
 \right|: \zeta \in S_{2K+1}, |\zeta -x| < \varepsilon  \right\}. $$
For each  $\zeta \in \gamma_z$ we have the estimate
\beq \label{beta}
|\beta(\zeta)| \le \alpha( |\zeta -x| ) \le  \alpha( \tfrac{3}{2}|z -x| ).
\eeq

By Cauchy's formula for derivatives,
\begin{eqnarray*}
f^{(n)}(z) &=&\frac{n!}{2 \pi \ii} \int_{\gamma_z} ~\frac{f(\zeta)}{(\zeta -z)^{n+1}} ~d \zeta\\
&=&\frac{n!}{2 \pi \ii} \int_{\gamma_z} ~\frac{
 a^0 + a^1 (\zeta -x) + \dots  + a^n (\zeta -x)^{n} + \beta(\zeta)(z-x)^{n}}{(\zeta -z)^{n+1}} ~d \zeta\\
&=&\frac{n!}{2 \pi \ii} \int_{\gamma_z} ~\frac{
 a^0}{(\zeta -z)^{n+1}} ~d \zeta + \frac{n!}{2 \pi \ii} \int_{\gamma_z}\frac{a^1 (\zeta -x)}{(\zeta -z)^{n+1}} ~d \zeta + \dots  \\
& ~ & ~~~ + \frac{n!}{2 \pi \ii} \int_{\gamma_z}\frac{a^n (\zeta -x)^{n}}{(\zeta -z)^{n+1}} ~d \zeta + \frac{n!}{2 \pi \ii} \int_{\gamma_z}\frac{\beta(\zeta)(\zeta-x)^{n}}{(\zeta -z)^{n+1}} ~d \zeta \\
&=& n! a^n +  \frac{n!}{2 \pi \ii} \int_{\gamma_z}\frac{\beta(\zeta)(\zeta-x)^{n}}{(\zeta -z)^{n+1}} ~d \zeta.\\
\end{eqnarray*}
By  (\ref{on_gamma_z})  and (\ref{beta}),
the integrand in the preceding integral is bounded in modulus by
\begin{eqnarray*}
 \left| \frac{\beta(\zeta)(\zeta-x)^{n}}{(\zeta -z)^{n+1}} \right|
& \le & \alpha( \tfrac{3}{2}|z -x|)\frac{((K +1) \;\dist (z, \R))^n}{(\tfrac{1}{2}\; \dist (z, \R))^{n+1}}\\
&= &\alpha( \tfrac{3}{2}|z -x|) \;2^{n+1} (K+1)^n \;(\dist (z, \R))^{-1}.
\end{eqnarray*}
for all  $\zeta \in \gamma_z$. The length of $\gamma_z$ is equal to $\pi \dist (z, \R)$, and so the integral itself is bounded in modulus by 
$\alpha( \tfrac{3}{2}|z -x|)\; 2^{n+1} \pi (K+1)^n $
which tends to $0$ as $z \to x$ in the region $S_K$. This proves that 
\[
\lim_{z \nt x}\frac{f^{(n)}(z)}{n!} = a^n. 
\]
\end{proof}

\section{Julia reduction and augmentation in the Pick class} \label{reduction}

In 1920 G. Julia  \cite{Ju20}, in the course of proving the well known ``Julia's Lemma" for bounded analytic functions on the disc, introduced a technique for passing from a function in the Pick class to a simpler one and back again. He showed that if $f\in\Pick$ is analytic at $x$ then the reduction of $f$ at $x$ also belongs to 
$\Pick$. Subsequently Nevanlinna \cite{Nev1} proved that 
the conclusion remains under a weaker hypothesis than analyticity at $x$. 
Whereas our earlier paper \cite{ALY10} needed only Julia's result, the present one depends crucially on Nevanlinna's (considerably more subtle) refinement.

We shall say that $ x \in \R$ is  a {\em $B$-point for} $f \in \Pick$ if the Carath\'eodory condition
\beq \label{cc}
\liminf_{z\to x} \frac{\im f(z)}{\im z} < \infty
\eeq
holds (\cite{AMcCY10}).  A part of the Carath\'eodory-Julia Theorem \cite{car54, S2} asserts that if $x\in\R$ is a $B$-point for $f\in\Pick$ then $f$ has a nontangential limit and an angular derivative at $x$.  We shall denote these quantities by $f(x), \ f'(x)$ respectively.  The theorem also tells us that $f'(x)>0$ if $f$ is not a constant function.

\begin{definition} \label{defreduce} \rm
 (1)  For any non-constant function $f\in\Pick$ and any $x\in\R$ such that $x$ is a $B$-point for $f$ we define the {\em reduction of $f$ at $x$} to be the function $g$ on $\Pi$ given by the equation
\beq \label{reducef}
g(z) = -\frac{1}{f(z)-f(x)} + \frac{1}{f'(x)(z-x)}.
\eeq

 (2)  For any $g\in\Pick$, any $x\in\R$  and any $a_0\in\R, a_1 > 0$, we define the {\em augmentation of $g$ at $x$ by $a_0, a_1$} to be the function $f$ on $\Pi$ given by
\beq \label{augmentg}
 \frac{1}{f(z)-a_0} =  \frac{1}{a_1(z-x)} -g(z).
\eeq
\end{definition}
Note that in (1), since $f(x)$ is real and $f$ is non-constant, the denominator $f(z)-f(x)$ is non-zero, by the maximum principle.
Furthermore $f$ defined by equation (\ref{augmentg}) is necessarily non-constant, for otherwise
\[
\im g(z) = \mathrm{const} + \frac{1}{a_1}\im \frac{1}{z-x},
\]
and the last term can be an arbitrarily large negative number for $z\in\Pi$, contrary to the choice of $g\in\Pick$.

Here are the crucial invariance properties of reduction and its inverse.
\begin{theorem} \label{propfg} 
Let $x \in\R$.
\begin{enumerate}
\item[\rm(1)] If  $x$ is a $B$-point for a non-constant function $f\in\Pick$ then the reduction $g$ of $f$ at $x$ also belongs to $\Pick$.
\item[\rm(2)] If $g\in\Pick$ and $a_0\in\R,\, a_1>0$ then the augmentation $f$ of $g$ at $x$ by $a_0,\, a_1$  belongs to $\Pick$, has a $B$-point at $x$  and satisfies $f(x)=a_0, \ f'(x) \leq a_1$.
Moreover
\beq\label{nopole}
 f'(x) = a_1  \quad\mbox{ if and only if }\quad \lim_{y\to 0+} yg(x+\ii y) =0.
\eeq
\end{enumerate}
\end{theorem}
\begin{proof}
Nevanlinna proved the analogue of (1) for the case that $x = \infty$, but his proof is easily modified for finite $x$; details are in 
\cite[Theorem 5.4]{AMcCY10}. Here is a bare outline. Let $a^1=f'(x)>0$. One shows that, for any $\varepsilon > 0$, 
\beq\label{expgro}
-\im g(z) \leq \frac{\varepsilon}{a_1} |w|
\eeq
for all $w \in \Pi$ of sufficiently large modulus, where
$w = -1/(z-x)$.
Introduce the analytic function $F$ on $\Pi$ by
\[
F(w) = \e^{\ii g(z)} = \e^{\ii g(x-1/w)}.
\]
We have, for any $w \in \Pi$,
\[
|F(w)| = \e^{\re \ii g(z)} = \e^{-\im g(z)}.
\]
By inequality \ref{expgro}, $F$ has only exponential growth on $\Pi$. Apply the Phragm\'en-Lindel\"of Theorem (e.g. \cite[p. 218]{BakNew}) to show that  $|F| \leq \e^{\de/a_1}$ on $\Pi+\ii\de$, for any $\delta > 0$.  On letting $\de$ tend to zero we deduce that $|F| \leq 1$ on $\Pi$, and hence that $\im g \geq 0$ on $\Pi$. Thus $g\in\Pick$.

The proof of (2) is an exercise in the mapping properties of linear fractional transformations of the complex plane. Again, details are in \cite{AMcCY10}. 
\end{proof}

\begin{remark} {\rm Let $g$ be a real rational function of degree $m$ and let $f$ be 
the augmentation of $g$ at $x$ by $a^0, a^1>0$. Then $f$ is a real rational function of degree $m+1$.}
\end{remark}

Here, as usual, the degree of a rational function $f = \frac{p}{q}$ is defined to be the maximum of the degrees of $p$ and $q$, where $p$, $q$ are polynomials in their lowest terms.

Pseudo-Taylor expansions behave well with respect to reduction and augmentation, as we now show.
\begin{proposition}\label{2augment_weak} {\rm (1)} Let $g,G  \in \Pick$ and $x\in\R$.  Suppose that  $G$ is analytic at $x$ and that, for some non-negative integer $N$,
\[
g(z) - G(z) = \ont((z-x)^N) \qquad \mbox{ as } z \to x. 
\]
Then the augmentations $f, F$ of  $g$, $G$ respectively at $x$ by $a^0 \in \R$ and $a^1 >0$ satisfy 
\beq\label{fF}
f(z) -F(z) = (g(z) - G(z))(F(z)-a^0)(f(z)-a^0)
\eeq
and $f(z) -F(z) = \ont((z-x)^{N+2})$ as $z \to x$.

{\rm (2)} Let $f\in\Pick$ be a non-constant function, let $x\in\R$ be a $B$-point for $f$ and let $F$ be a polynomial such that 
\beq\label{fFoN}
f(z) -F(z) = \ont((z-x)^{N})
\eeq
for some $N \ge 2$. Let $g$, $G$ be the reductions of $f, F$ respectively at $x$. Then $F'(x) >0$ and 
\beq\label{simpleid}
f(z) -F(z) = (g(z) - G(z))(F(z)-F(x))(f(z)-F(x))
\eeq
for all $z \in \Pi$. Moreover, $g$ has a pseudo-Taylor expansion at $x$ of order $N-2$, and
\beq\label{gGo(N-2)}
g(z) -G(z) = \ont((z-x)^{N-2}).
\eeq
\end{proposition}
\begin{proof} (1) Note that
\begin{eqnarray*}
g(z) - G(z) &=& -\frac{1}{f(z)-a^0} + \frac{1}{a^1 (z-x)} - \left(-\frac{1}{F(z)-a^0} + \frac{1}{a^1 (z-x)} \right)\\
&=& \frac{f(z) -F(z)}{(F(z)-a^0)(f(z)-a^0)}.
\end{eqnarray*}
Therefore
\begin{eqnarray*}
f(z) -F(z)  &=& (g(z) - G(z))(F(z)-a^0)(f(z)-a^0)\\
 &=& \ont((z-x)^N)\ont((z-x))\ont((z-x))= \ont((z-x)^{N+2})
\end{eqnarray*}
 as $z \to x$.\\

(2)  By the  Carath\'eodory-Julia Theorem, the nontangential limit and angular derivative
$ f(x)\in\R$ and $f'(x)>0$ exist.  By equation (\ref{fFoN}), $F(x)=f(x)$ and $F'(x) =  f'(x).$  Hence
\[
F(z)= f(x) + f'(x)(z-x) +o(z-x).
\]
The identity (\ref{simpleid}) is immediate as in Part (1), and we have
\begin{eqnarray*}
g(z) - G(z) &=& \frac{f(z) -F(z)}{(F(z)-f(x))(f(z)-f(x))}\\
&=& \frac{\ont((z-x)^{N})}{[f'(x) (z-x) + \ont(z-x)]^2}\\
&=& \ont((z-x)^{N-2}).
\end{eqnarray*}
Since $G$ is the reduction of a polynomial $F$, it is rational and is analytic at $x$. Thus it has an infinite Taylor expansion about $x$, and so $g$ has a pseudo-Taylor expansion of order $N-2$ about $x$.
\end{proof}

\begin{corollary}\label{fg_C-point_N} Let $x \in \R$, let $f\in\Pick$ be non-constant function and let $g$ be the reduction of $f$ at $x$. Let $N > 2$. Then $f$ has a pseudo-Taylor expansion of order $N$ about $x$  if and only if $g$ has a pseudo-Taylor expansion of order $N-2$ about $x$.
\end{corollary}
\begin{proof}
It follows from Proposition \ref{2augment_weak}.
\end{proof}

\begin{lemma}\label{2augment_relax_weak} Let  $x\in \R$ and 
let $f\in\Pick$ be a non-constant function. Suppose $f$ has a pseudo-Taylor expansion of order $N \ge 2$ at $x$ 
 and let $F$ be a polynomial such that 
\beq\label{fFoN_relax}
f(z) -F(z) = A (z-x)^N + \ont((z-x)^{N})
\eeq
for some $A \in \C$. Let $g$, $G$ be the reductions of $f, F$ respectively at $x$. Then $F'(x) >0$ and 
\beq\label{gG_relax}
g(z) -G(z) = \frac{A}{F'(x)^2} (z-x)^{N-2} + \ont((z-x)^{N-2}).
\eeq
\end{lemma}
\begin{proof} As in Proposition \ref{2augment_weak}, $F(x)=f(x)$ and $F'(x)=f'(x) >0$. The
equation (\ref{simpleid}) implies
\begin{eqnarray*}
g(z) - G(z) &=& \frac{f(z) -F(z)}{(F(z)-F(x))(f(z)-f(x))}\\
&=& \frac{A (z-x)^N + \ont((z-x)^{N})}{[F'(x) (z-x) + \ont((z-x))]^2}\\
&=& \frac{A}{F'(x)^2} (z-x)^{N-2} +\ont((z-x)^{N-2}).
\end{eqnarray*}
The relation (\ref{gG_relax}) follows. 
\end{proof}

Another ingredient of the proof of our main result is an identity for Hankel matrices, which shows that the reduction of power series corresponds to Schur complementation of Hankel matrices.

\begin{theorem} \label{congruent}
 Let
\[
f=\sum_{j=0}^\infty f_j z^j, \qquad g=\sum_{j=0}^\infty g_j z^j
\]
 be formal power series over $\C$ with $f_1\neq 0$ and $f_0,f_1,\dots,f_n \in\R$, and let  $g$ be the reduction of $f$ at $0$.  Then the $n\times n$ Hankel matrix 
$$
H_n(g)=[g_{i+j-1}]_{i,j=1}^n
$$ 
is congruent to the Schur complement of the $(1,1)$ entry in the $(n+1)\times(n+1)$ Hankel matrix 
$$
H_{n+1}(f)=[f_{i+j-1}]_{i,j=1}^{n+1}.
$$ 
 Consequently $H_{n+1}(f)> 0$ if and only if $f_1>0$ and $H_n(g)>0$.
\end{theorem}
This is  Corollary 3.4  of \cite{ALY10}.  It is convenient to introduce some notation for the relationship described in the theorem.  For any $n \times n$ matrix $A = [a_{ij}]$ with $a_{11} \neq 0$ we define
$\schur A$ to be the Schur complement of $[a_{11}]$ in $A$. Thus, for $f$ and $g$ as in Theorem \ref{congruent}, $\schur H_{n+1}(f)$ is congruent to $H_n(g)$.

\section{A relaxation of the boundary Carath\'{e}odory-Fej\'{e}r problem} \label{relax}

Solvability of Problem $\partial CF\Pick(\R)$ is best approached through a slight relaxation of the problem, in which the final interpolation condition ($f^{(n)}(x)/n!=a^n$) is replaced by an inequality (see for example \cite{Geo98,BD,ALY10}).  The reason is that solvability of the relaxed problem, not the original one, corresponds to positivity of a Hankel matrix.  We therefore consider:

\noindent {\bf Problem $\partial CF\Pick'(\R)$} \quad {\em
Given a point $x\in \R$ and $a^0,a^1,\dots,a^n \in \R$, find a function $f$ in the Pick class such that $f$ is analytic at  $x$,
\beq \label{interpcondRel}
\frac{f^{(k)}(x)}{k!} = a^k, \qquad k=0,1,\dots,n-1, \;\;{\text and } \;\;  \frac{f^{(n)}(x)}{n!} \le a^n.\\
\eeq
}

The terminology for the problem was introduced in \cite{ALY10}, but in this paper we are interested in functions $f$ that satisfy the interpolation conditions in a weak sense.
We define a function $f \in \Pick$ to be a {\em weak solution} of Problem $\partial CF\Pick'(\R)$ if $f$ has a pseudo-Taylor expansion of order $n$ at $x$ and the $j$th pseudo-Taylor coefficient of $f$ at $x$ is $a^j$ for $j=0,1,\dots,n-1$ and is no greater than $a^n$ for $j=n$.  

Here is an alternative description of weak solutions. Suppose that $F$ is analytic at $x$ and
\[
F(z) = a^0 + a^1 (z-x) + \dots + a^{n} (z-x)^{n} + O((z-x)^{n+1}).
\]
Then a function $f \in \Pick$ is a weak solution of 
 Problem $\partial CF\Pick'(\R)$ if and only if, for some $A \le 0$,
\beq \label{fFAo}
f(z) -F(z) = A(z-x)^n + \ont ((z-x)^{n}).
\eeq

We say a function $f \in \Pick$ is a {\em nontangential solution} of Problem $\partial CF\Pick'(\R)$ if 
\beq \label{quasi-solution}
\lim_{z \nt x}\frac{f^{(k)}(z)}{k!} = a^k, \qquad k=0,1,\dots,n-1, \;\;\text{ and } \;\;  \lim_{z \nt x} \frac{f^{(n)}(z)}{n!} \le a^n,
\eeq
and we define a {\em radial solution} of Problem $\partial CF\Pick'(\R)$ in the obvious way.

We might expect that more problems $\partial CF\Pick'$ would admit weak solutions than true solutions. In fact, though, the crux of the problem is the analytic case. This assertion is justified by the following result.

Corresponding to the sequence $a=(a^0, a^1, \dots, a^n)$  and any positive integer $m$ such that $2m-1 \leq n$ we define the Hankel matrix $H_m(a)$ to be the $m\times m$ matrix
$[a_{i+j-1}]_{i,j=1}^m$.  If $F$ is a function analytic at the interpolation node $x$, we shall write $H_m(F)$ to mean $H_m(f_0, \dots, f_{2m-1})$, where $f_j$ is the $j$th Taylor coefficient of $F$ at $x$.

\begin{theorem} \label{weakequiv} Let $n$ be an odd positive integer.
Then the following  statements are equivalent:
\begin{enumerate}
\item[\rm(1)]  Problem $\partial CF\Pick'(\R)$ has a weak solution;
\item[\rm(2)]  Problem $\partial CF\Pick'(\R)$ has a solution which is analytic at $x$;
\item[\rm(3)]  Problem $\partial CF\Pick'(\R)$ has a rational solution;
\item[\rm(4)]  Problem $\partial CF\Pick'(\R)$ has a real rational solution;
\item[\rm(5)]  Problem $\partial CF\Pick'(\R)$ has a nontangential solution;
\item[\rm(6)]  Problem $\partial CF\Pick'(\R)$ has a radial solution;
\item[\rm(7)]  The Hankel matrix $H_m(a)$ is positive, where $m = \tfrac 12 (n+1)$.
\end{enumerate} 
\end{theorem}

\begin{proof} 
By \cite[Theorem 6.1]{ALY10},
(2) $\Leftrightarrow$ (3) $\Leftrightarrow$ (4) $\Leftrightarrow$ (7), and obviously (4)$\Rightarrow$(1) and (4)$\Rightarrow$(5). By Theorem \ref{Taylor}, (5)$\Leftrightarrow$(6)$\Leftrightarrow$(1).

We must show that (1)$\Rightarrow$(7).   Suppose that $f$ is a weak solution of Problem $\partial CF\Pick'(\R)$, with pseudo-Taylor expansion
\[
f(z)=\sum_{j=0}^n f_j(z-x)^j +\ont((z-x)^n).
\]
Thus $f_j=a^j$ for $j\leq n-1$ and $f_n \leq a_n$.
We can assume that $f$ is nonconstant.

 Consider the case that $m=1=n$.
We have  $\im f(x+\ii y) = a^1y + o(y)$, and hence
\[
\lim_{y \to 0+} \frac{\im f(x+iy)}{y} =   a^1 < \infty.
\]
It follows from the Carath\'{e}odory-Julia theorem \cite{car54} that $a^1> 0$, which is to say that $H_1(a) > 0$. Thus 
 (1)$\Rightarrow$(7) when $m=1$.

Now consider $m \ge 2$ and suppose the implication (1)$\Rightarrow$(7)
valid for $m-1$.   
Let
\[
F(z)= \sum_{j=0}^n f_j(z-x)^j,
\]
so that $f(z)-F(z)=\ont((z-x)^n)$.
Let $g, G$ be the reductions of $f, F$ respectively at $x$; then $g\in\Pick$.  By Proposition \ref{2augment_weak}, $g$ has a pseudo-Taylor expansion of order $n-2=2m-3$ about $x$ and
\beq\label{g-G}
 g(z)-G(z) =    \ont((z-x)^{2m-3}).
\eeq
That is to say, $g$ is a weak solution of Problem $\partial CF\Pick'(\R)$ with data $G$ and with new ``$n$" equal to $n-2=2m-3$.  Accordingly this last problem has a weak solution, and we may invoke the inductive hypothesis to assert that $H_{m-1}(G) \geq 0$.  

By the Hankel identity, Theorem \ref{congruent}, 
$ H_{m-1}(G)$ is congruent to $\schur H_{m}(F)$.  Since (again by the Carath\'{e}odory-Julia theorem) $f_1= a^1 >0$, it follows that $ H_{m}(F) \ge 0$.
Now $H_m(a)$ and $H_m(F)$ differ only in their southeast corner entries -- in fact
\[
H_m(a) = H_m(F) + \diag \{0,0,\dots,a^n -f_n\}  \geq H_m(F) \ge 0.
\]
Thus (1)$\Rightarrow$(7), and the theorem follows by induction. 
 \end{proof}
We now consider the question of determinacy for Problem $\partial CF\Pick'(\R)$.  In the analytic case, the problem is determinate if and only if the associated Hankel matrix is positive and singular \cite[Theorem 5.1]{ALY10}.  In principle, there might be one analytic solution and many weak solutions of a problem, but in fact this does not happen.  
\begin{theorem} \label{probprimeHD_weak} 
Let $ x\in\R$, $a=(a^0,\dots,a^{2m-1})\in \R^{2m}$ for some $m \ge 1$. Problem $\partial CF\Pick'(\R)$ has a unique weak solution if and only if  the associated Hankel matrix $H_m(a)$ is positive and singular.
\end{theorem}
\begin{proof} By \cite[Theorem 5.1]{ALY10}, if $H_m(a) > 0$ the Problem $\partial CF\Pick(\R)$ is indeterminate. Thus,  by Theorem \ref{weakequiv}, necessity holds.

Suppose that  $H_m(a)$ is positive and singular. We show that  
Problem $\partial CF\Pick'(\R)$ has a unique weak solution.

Consider the case $m=1$. Here  $a^1=0$, and the constant function equal to $a^0$ is a solution of  Problem $\partial CF\Pick'(\R)$. 
Let $f$ be any weak solution, so that $f \in \Pick$ and
\beq\label{m=1}
 f(z) = a^0 + \ont ((z-x)).
\eeq
We have 
$ \im f(x+\ii y)/y \to 0$ as $y \to 0+$. Hence
\beq \label{B_cond}
\alpha \stackrel{\rm def}{=} \liminf_{z\to x, z \in \Pi} \frac{\im f(z)}{\im z}\leq \lim_{y\to 0+} \frac{\im f(x+\ii y)}{y} =0.
\eeq
By the Carath\'{e}odory-Julia theorem \cite{car54}, if $f$ is nonconstant then $\alpha > 0$. Thus the only weak solution is the constant $a^0$. The assertion of the theorem is therefore true when $m=1$.

Suppose the assertion holds for some $m \ge 1$; we prove it holds for $m+1$.

Let $H_{m+1}(a)$ be positive and singular for some $a=(a^0,\dots,a^{2m+1})$. Let  $F(z)=\sum_0^{2m+1} a^j (z-x)^j$.
 Assume that functions
$f_1$ and $f_2$ in $\Pick$ are solutions of the problem $\partial CF\Pick'(\R)$ with data $x$ and $a$. Then, for some $A_1, A_2 \le 0$,
\beq\label{f1f2FoN_relax}
f_i(z) -F(z) = A_i(z-x)^{2m+1} + \ont((z-x)^{2m+1}),\;\; i=1,2.
\eeq
Let $g_1, g_2$, $G$ be the reductions of $f_1, f_2$, $F$ respectively at $x$. Then $g_1, g_2 \in \Pick$ and $G$ is a rational function that is analytic at $x$. By Lemma \ref{2augment_relax_weak},
\beq\label{g1g2G_relax}
g_i(z) -G(z) = \frac{A_i}{(a^1)^2} (z-x)^{2m-1} + \ont((z-x)^{2m-1}), \;\; i=1,2.
\eeq
Since $\frac{A_i}{(a^1)^2} \le 0$, it follows that $g_1$ and $g_2$ are weak solutions of Problem $\partial CF\Pick'(\R)$ with data $x$, $b$ where $b = (b^0,\dots,b^{2m-1})$ comprises the first $2m$ Taylor coefficients of $G$ about $x$.
By Theorem \ref{congruent},  the associated Hankel matrix of this problem,  $H_m(b)$ is congruent to  $\schur H_{m+1}(a)$. Since $H_{m+1}(a)$ is positive and singular, so is  $H_m(b)$. By the inductive hypothesis, the problem has a unique solution, and so $g_1=g_2$.  Since $f_1, \ f_2$ are both equal to the augmentation of this function $g_1=g_2$ at $x$ by  $a^0$, $a^1$ we have $f_1=f_2$. Thus, by induction, the statement of Theorem \ref{probprimeHD_weak} holds for all $m\geq 1$.
\end{proof}

\section{Weak solutions of Problem $\partial CF\Pick$} \label{weak}

In this section we prove the main result of the paper, a criterion for the existence of a weak solution of $\partial CF\Pick(\R)$.  As in \cite{ALY10} we deduce the result from the corresponding criterion for Problem $\partial CF\Pick'(\R)$, but there is a subtlety: the deduction depends on the condition for the uniqueness of solutions of Problem $\partial CF\Pick'(\R)$, and now this must be understood in the sense of uniqueness in the class of weak solutions.  We shall therefore need to use Theorem \ref{probprimeHD_weak} above.

We shall say that the Hankel matrix  $H_{m}(a)$ 
is {\em southeast-minimally positive} if $H_{m}(a) \ge 0$ and, for every  $\varepsilon >0$, $H_{m}(a) - {\rm diag}\{0,0,\dots, \varepsilon\}$ is not positive.
We shall abbreviate ``southeast-minimally" to ``SE-minimally".

\begin{theorem} \label{main_theorem}
Let $n$ be an odd positive integer and let $a=(a^0,\dots,a^n)\in \R^{n+1}$. 
The following  statements are equivalent:
\begin{enumerate}
\item[\rm(1)]  Problem $\partial CF\Pick(\R)$ has a weak solution;
\item[\rm(2)] Problem $\partial CF\Pick(\R)$ has a nontangential solution;
\item[\rm(3)] Problem $\partial CF\Pick(\R)$ has a radial solution;
\item[\rm(4)]  Problem $\partial CF\Pick(\R)$ has a solution which is analytic at $x$;
\item[\rm(5)] the associated Hankel matrix $H_{m}(a)$, $n=2m -1$, is either positive definite or SE-minimally positive. 
\end{enumerate} 

Moreover, the problem has a {\em unique} weak solution if and only if $H_m(a)$ is SE-minimally positive, and in this case the solution is rational of degree equal to $\rank H_m(a)$.

\end{theorem}
\begin{proof} 
By \cite[Theorem 7.1]{ALY10},
(4) $\Leftrightarrow$ (5). By Theorem \ref{Taylor}, (1) $\Leftrightarrow$ (2) $\Leftrightarrow$ (3). It is clear that (4)  $\Rightarrow$ (1). We will show that (1)  $\Rightarrow$ (5). 

(1) $\Rightarrow$ (5). Suppose that Problem $\partial CF\Pick(\R)$ has a weak solution $f\in\Pick$ but that its  Hankel matrix $H_m(a)$ is neither positive definite nor SE-minimally positive.  {\em A fortiori} $f$ is a weak solution of Problem $\partial CF\Pick'(\R)$, and so, by Theorem \ref{weakequiv}, $H_m(a)\geq 0$.  Since $H_m(a)$ is not positive definite, $H_m(a)$ is singular, and so, by Theorem \ref{probprimeHD_weak}, Problem $\partial CF\Pick'(\R)$ has the {\em unique} weak solution $f$.  Since $H_m(a)$ is not SE-minimally positive there is some positive $a^{n}{'} < a^{n}$ such that $H_m(f)\geq 0$, where $H_m(f)$ is the matrix obtained when the $(m,m)$ entry $a^{n}$, $n=2m-1$, of $H_m(a)$ is replaced by $a^{n}{'}$.  Again by Theorem \ref{weakequiv}, there exists $h\in\Pick$  such that 
$$
\lim_{z \nt x}\frac{h^{(k)}(z)}{k!} = a^k, \qquad k=0,1,\dots,n-1, \;\;\text{ and } \;\;  \lim_{z \nt x} \frac{h^{(n)}(z)}{n!} \le a^{n}{'}< a^{n}.
$$
 In view of the last relation  we have $h\neq f$, while clearly $h$ is a weak solution of Problem $\partial CF\Pick'(\R)$, as is $f$.  This contradicts the uniqueness of the weak solution $f$.  Hence if the problem is solvable then either $H_m(a)>0$ or $H_m(a)$ is SE-minimally positive.
\end{proof}

We note that a different solvability criterion is given by D. Georgijevi\'c in \cite{Geo98}: Problem $\partial CF\Pick(\R)$ is solvable if and only if $H_m(a) \geq 0$ and its rank is equal to the rank of each of its singular submatrices.  His methods are quite different from ours.

There is also a version of Theorem \ref{main_theorem} for even $n$.
\begin{theorem}\label{main_theorem_even}
Let $n$ be an even positive integer.
 the following  statements are equivalent:
\begin{enumerate}
\item[\rm(1)]  Problem $\partial CF\Pick(\R)$ has a weak solution;
\item[\rm(2)] Problem $\partial CF\Pick(\R)$ has a nontangential solution;
\item[\rm(3)] Problem $\partial CF\Pick(\R)$ has a radial solution;
\item[\rm(4)]  Problem $\partial CF\Pick(\R)$ has a solution which is analytic at $x$;
\item[\rm(5)] either the associated Hankel matrix $H_{m}(a)$, $n=2m$, is positive definite or both $H_{m}(a)$ is SE-minimally positive and $a^{n}$ satisfies 
\beq \label{form_a_2m_Main}
a^{n} = \left[ \begin{array}{cccc}
  a^m & a^{m+1} & \dots &a^{m+r-1}\end{array}\right] H_r(a)^{-1}
\left[\begin{array}{c} a^{m+1} \\ a^{m+2}\\ \cdot \\ a^{m+r} \end{array}\right] 
\eeq 
where $r=\rank H_m(a)$.
\end{enumerate} 

Moreover, the problem has a {\em unique} solution if and only if $H_m(a)$ is SE-minimally positive and $a^{n}$ satisfies equation {\rm (\ref{form_a_2m_Main})}.
\end{theorem}
\begin{proof}
By \cite[Theorem 7.1]{ALY10},
(4) $\Leftrightarrow$ (5). By Theorem \ref{Taylor}, (1) $\Leftrightarrow$ (2) $\Leftrightarrow$ (3). It is clear that (4)  $\Rightarrow$ (2). We will show that (2)  $\Rightarrow$ (5).

(2) $\Rightarrow$ (5).  Suppose that Problem $\partial CF\Pick(\R)$ has a weak solution $f\in\Pick$  such that $\lim_{z \nt x} f^{(k)}(z)/k! = a^k$ for $k=0,1,\dots,2m$. This 
$f\in\Pick$ is also a weak solution of Problem $\partial CF\Pick(\R)$ for $n= 2m-1$.
The  Hankel matrix $H_m(a)$ for Problem $\partial CF\Pick(\R)$ with $n=2m$ and with $n= 2m-1$ is the same.
By Theorem \ref{main_theorem}, $H_m(a)$ is  positive definite or SE-minimally positive. 

In the case that $a^1=0$, the constant function $f(z)=a^0$ is the solution of $\partial CF\Pick(\R)$. Therefore, $a^2 =a^3= \dots = a^{2m}=0$. Thus $H_m(a)$ is  SE-minimally positive and $a^{n}$ satisfies (\ref{form_a_2m_Main}).

If $H_m(a)$ is  SE-minimally positive and $a^1>0$ then by \cite [Proposition 7.4]{ALY10}, $a^{n}$ satisfies (\ref{form_a_2m_Main}). 
\end{proof}

\begin{remark} \rm
In \cite[Theorem 8.3]{ALY10} we gave a parametrization of all solutions of Problem $\partial CF\Pick(\R)$ in the indeterminate case.  The parametrization expresses the general solution $f$ as a continued fraction, containing as parameter a free function $f_{m+1} \in\Pick$ that is analytic at $x$ (when $n=2m-1$).  It is simple to modify this parametrization to describe all {\em weak} solutions of Problem $\partial CF\Pick(\R)$: one simply takes the parameter set to be the set of all $f_{m+1} \in\Pick$ such that 
$\lim_{y\to 0+} yf_{m+1}(x+\ii y) =0$, with no requirement of analyticity at $x$.    This is essentially Nevanlinna's parametrization \cite[Satz I, p. 11]{Nev1922}. There is a similar parametrization in the case of even $n$.
\end{remark}

We conclude with an observation about a natural generalization of our main theorem.
Since functions in $\Pick$ can have simple poles with negative residue at points of the real axis, it is natural to study a slightly more general problem  than $\partial CF\Pick(\R)$, in which the $(-1)$th Laurent coefficient is also prescribed \cite{Geo98, ALY10}: \\

{\em Given $x\in\R$ and $a^{-1}, a^0, \dots, a^n \in \R$ with $a^1>0$, determine whether there is a function $f\in\Pick$ such that}
\beq\label{genweak}
f(z) = \frac{a^{-1}}{z-x} + a^0+a^1 (z-x)+ \dots+a^n (z-x)^n + \ont((z-x)^n).
\eeq
 
In fact this interpolation problem is equivalent to the problem  $\partial CF\Pick(\R)$ obtained by simply suppressing the condition on the $(-1)$th Laurent coefficient.
\begin{proposition}
There exists $f\in\Pick$ such that equation {\rm (\ref{genweak})} holds if and only if $a^{-1} \leq 0$ and there exists an $F\in\Pick$ such that
\beq\label{Fsolves}
F(z) =  a^0+a^1 (z-x)+ \dots+a^n (z-x)^n + \ont((z-x)^n).
\eeq
\end{proposition}
\begin{proof}
Sufficiency is easy: if $a^{-1} \leq 0$ and $F\in\Pick$ satisfies (\ref{Fsolves}) then the function
\[
f(z) = F(z) + a^{-1}/(z-x)
\]
belongs to $\Pick$ and satisfies (\ref{genweak}).

Conversely, suppose $f\in\Pick$ satisfies (\ref{genweak}), and let $F(z) = f(z)-a^{-1}/(z-x)$.  Certainly $F$ satisfies (\ref{Fsolves}); our task is to show that $a^{-1} \leq 0$ and $F\in\Pick$.  Since $f\in\Pick$, we have for any $y> 0$,
\[
0\leq \im  f(x+\ii y) = \im \frac{a^{-1}}{\ii y} +o(1)= -\frac{a^{-1}}{y} + o(1),
\]
and hence $a^{-1} \leq 0$.

Observe that $0$ is a $B$-point for the function $-1/f$ (which lies in $\Pick$) if and only if
\beq\label{li}
\liminf_{z\to x} \frac{\im f(z)}{|f(z)|^2 \im z} < \infty,
\eeq
a relation which does hold, in view of the fact that $f(z)= a^{-1}/(z-x) +\ont(1)$, and we find that the $\liminf$ (\ref{li}) is $-1/a^{-1}$.  Thus $-1/f$ has nontangential limit $0$ and angular derivative $-1/a^{-1}$ at $x$.  Let $G$ be the reduction of $-1/f$ at $0$.  By Theorem 3.2(1), Nevanlinna's refinement of Julia's lemma, $G\in\Pick$.  But
\[
G(z)= -\frac{1}{-1/f(z)} +\frac{1}{(-1/a^{-1})(z-x)} = f(z) - \frac{a^{-1}}{z-x} = F(z).
\]
Thus $F\in\Pick$.
\end{proof}
\begin{corollary}
Let $n$ be an odd positive integer, $n=2m-1$, and let $a^{-1}, a^0, \dots, a^n \in\R$ with $a^1>0$.  There exists a function $f$ in $\Pick$ such that equation {\rm (\ref{genweak})} holds if and only if $a^{-1} \leq 0$ and the Hankel matrix $H_m(a)$ is either positive definite or SE-minimally positive.
\end{corollary}

\section{Appendix} \label{appendix}
Here we justify the assertions made concerning Example \ref{ex2}.  We show that, for any integer $\nu \geq 4$, the function
\beq \label{deffnu}
f_\nu(z) = -\sum_{k=1}^\infty \frac{1}{k^\nu z + k^{\nu-1}}
\eeq
belongs to the Pick class, has a pseudo-Taylor expansion of order $\nu-3$ given by equation (\ref{expandfnu}) and has no expansion of order $\nu-2$.

The series (\ref{deffnu}) converges locally uniformly in $\Pi$, and so $f$ is analytic in $\Pi$. Since each summand belongs to $\Pick$, so does $f_{\nu}$. 
For $j = 0,1, \dots, \nu -3$, we obtain, with the aid of  the Dominated Convergence Theorem,
\begin{eqnarray*}
\lim_{y \to 0+} \frac{f_{\nu}^{(j)}(\ii y)}{j!}& =&
\lim_{y \to 0+} \sum_{k=1}^{\infty} 
\frac{(-1)^{j+1}}{k^{\nu} \left(\ii y + \tfrac{1}{k} \right)^{j+1}}\\
& =&(-1)^{j+1} \zeta(\nu - j -1).
\end{eqnarray*}
and so, by Theorem \ref{Taylor}, the expansion (\ref{expandfnu}) is indeed a pseudo-Taylor expansion of $f_\nu$ of order $\nu-3$.
However $f_{\nu}$ does not have a pseudo-Taylor expansion 
of order $\nu-2$.
In view of Theorem \ref{Taylor}, this claim will follow if we can show that $f_{\nu}^{(\nu-2)}(\ii y)$ does not have a finite limit as $y \to 0+$. We have
\[
\frac{f_{\nu}^{(\nu -2)}(\ii y)}{(\nu -2)!} = (-1)^{\nu -1} \sum_{k=1}^{\infty} h_{\nu}(k, y)
\]
where, for $t \ge 1$ and $y > 0$,
\[
h_{\nu}(t, y)= \frac{1}{t^{\nu}\left(\ii y + \tfrac{1}{t} \right)^{\nu -1}} = \frac{1}{t \left(1 + \ii t y \right)^{\nu -1}}.
\]
Now
\begin{eqnarray*}
\int_{1}^{\infty} h_{\nu}(t, y) dt &=& \int_{1}^{\infty}\frac{dt}{t \left(1 + \ii t y \right)^{\nu -1}} = \int_{y}^{\infty}\frac{du}{u \left(1 + \ii u \right)^{\nu -1}}\\
&=& - \sum_{j=1}^{\nu -2} \frac{1}{j \left(1 + \ii y \right)^{j}}
- \log \frac{\ii y}{1 + \ii  y} \to \infty
\end{eqnarray*}
as $y \to 0+$, while
\[
\left|\frac{\partial h_{\nu}}{\partial t} (t,y) \right| =
\frac{1}{t^2} \left|\frac{1+ \nu \ii t y}{ \left(1 + \ii t y \right)^{\nu}} \right|.
\]
Hence, for $t \in [k, k+1]$ and  $y > 0$,
\[
\left|\frac{\partial h_{\nu}}{\partial t} (t,y) \right| \le \frac{C_{\nu}}{k^2},
\]
where
\[
C_{\nu} = \sup_{y >0, t \ge 1} \left|\frac{1+ \nu \ii t y}{ \left(1 + \ii t y \right)^{\nu}} \right| =
\sup_{\tau >0} \left|\frac{1+ \nu^2 \tau}{ \left(1 + \tau \right)^{\nu}} \right|^{1/2} < \infty
\]
(in fact  $C_{\nu}^2 = {\nu}/{ \left(1 + \frac{1}{\nu} \right)^{\nu-1}}$). By the Mean Value Theorem, for $t \in [k, k+1]$,
\[
\left|h_{\nu}(t, y)- h_{\nu}(k, y)\right| \le \frac{C_{\nu}}{k^2},
\]
and so, for all $y > 0$,
\[
\left|\int_{1}^{\infty} h_{\nu}(t, y) dt  - \sum_{k=1}^{\infty} h_{\nu}(k, y) \right| \le \sum_{k=1}^{\infty} \frac{C_{\nu}}{k^2}
=\frac{C_{\nu} \pi^2}{6}.
\]
Hence
\[ 
\lim_{y\to 0+} f_\nu^{(\nu-2)} (\ii y) = (-1)^{\nu-1}(\nu-2)! \lim_{y \to 0+} \sum_{k=1}^{\infty} h_{\nu}(k, y) = \infty.
\]
Thus, as claimed, $f_{\nu}$ does not have a pseudo-Taylor expansion 
of order $\nu-2$ at $0$.


JIM AGLER, Department of Mathematics, University of California at San Diego, CA \textup{92103}, USA\\

ZINAIDA A. LYKOVA,
School of Mathematics and Statistics, Newcastle University,
 NE\textup{1} \textup{7}RU, U.K.~~
e-mail\textup{: \texttt{Z.A.Lykova@newcastle.ac.uk}}\\

N. J. YOUNG, School of Mathematics, Leeds University, LS2 9JT, U.K.~~
e-mail\textup{: \texttt{N.J.Young@leeds.ac.uk}}

\end{document}